\documentclass[12pt]{amsart}
\usepackage{amsmath, amssymb, amsfonts, amsthm, latexsym} 
\usepackage{amsbsy}
\usepackage{graphicx}
\usepackage{tikz}
\usepackage{pdfpages}

\newtheorem{Theorem}{Theorem}[section]
\newtheorem{Lemma}[Theorem]{Lemma}
\newtheorem{Corollary}[Theorem]{Corollary}
\newtheorem{Proposition}[Theorem]{Proposition}

\theoremstyle{definition}
\newtheorem{Remark}[Theorem]{Remark}
\newtheorem{Example}[Theorem]{Example}

\def\pd{\operatorname{pd}}
\def\depth{\operatorname{depth}} 

\def\dstab{\operatorname{dstab}}
\def\supp{\operatorname{supp}}

\def\Min{\operatorname{Min}}
\def\Max{\operatorname{Max}}

\def\NN{{\mathbb N}}
\def\ZZ{{\mathbb Z}} 
\def\a{{\mathbf a}}
\def\b{{\mathbf b}}

\def\e{{\mathbf e}}
\def\1{{\mathbf 1}}
\def\mm{{\mathfrak m}}

\def\E{{\mathcal E}}
\def\F{{\mathcal F}}

\def\D{{\Delta}}

\begin{document}

\title{A general formula for the index\\ of depth stability of edge ideals}

\author{Ha Minh Lam}
\address{Institute of Mathematics, Vietnam Academy of Science and Technology, 18 Hoang Quoc Viet, 10307 Hanoi, Vietnam}
\email{hmlam@math.ac.vn}

\author{Ngo Viet Trung}
\address{Institute of Mathematics, Vietnam Academy of Science and Technology, 18 Hoang Quoc Viet, 10307 Hanoi, Vietnam}
\email{nvtrung@math.ac.vn}

\author{Tran Nam Trung}
\address{Institute of Mathematics, Vietnam Academy of Science and Technology, 18 Hoang Quoc Viet, 10307 Hanoi, Vietnam}
\email{tntrung@math.ac.vn}

%\thanks{}

\begin{abstract} 
By a classical result of Brodmann, the function $\depth R/I^t$ is asymptotically a constant, i.e. there is a number $s$ such that $\depth R/I^t = \depth R/I^s$ for $t > s$. One calls the smallest number $s$ with this property the index of depth stability of $I$ and denotes it by $\dstab(I)$. This invariant remains mysterious till now.
The main result of this paper gives an explicit formula for $\dstab(I)$ when $I$ is an arbitrary ideal generated by squarefree monomials of degree 2. That is the first general case where one can characterize $\dstab(I)$ explicitly. The formula expresses $\dstab(I)$ in terms of the associated graph. The proof involves new techniques which relate different topics such as simplicial complexes, systems of linear inequalities, graph parallelizations, and ear decompositions. It provides an effective method for the study of powers of edge ideals.
\end{abstract}

\maketitle

%%%%%%%%%%%%%%%%%%%%%%%%

%\section*{Introduction}

Let $R$ be a polynomial ring over a field and $I$ a homogeneous ideal in $R$.
The invariant $\depth R/I$ is the maximal length of a regular sequence of homogeneous elements in $R/I$. 
It is closely related to the projective dimension of $R/I$ by the Auslander-Buchsbaum formula
$$\depth R/I + \pd R/I = \dim R.$$

Brodmann's result on the behavior of the function $\depth R/I^t$ \cite{Br} immediately leads to the problem whether one can characterize $\dstab(I)$ {\it explicitly} in terms of the ideal $I$. 
There are several attempts to estimate $\dstab(I)$ for particular classes of ideals \cite{BM, BaRa, CMS, ERT, Ha, HT, HH, HH2, HQ, HRV, HTr, Mo, TNT}. 
However, almost nothing is known about the nature of $\dstab(I)$, even when $I$ is a squarefree monomial ideal. Such ideals can be associated with combinatorial objects such as graphs, hypergraphs or simplicial complexes, which can be used to describe algebraic properties of the ideal. 

Let $G$ be a simple graph (i.e.~without loops and multiple edges) on the vertex set $V = [1,n]$. Let $R = k[x_1,...,x_n]$ be a polynomial ring over a field $k$. The {\it edge ideal} $I = I(G)$ of $G$ is generated by the monomials $x_ix_j$, $\{i,j\} \in G$.  
If $G$ is a connected non-bipartite graph, it has been known for a long time that $\displaystyle \lim_{t \to \infty}\depth R/I^t = 0$ \cite{CMS} before a formula for $\dstab(I)$ was found in \cite{LT}. 

If $G$ is a connected bipartite graph, the problem is more difficult because $\depth R/I^t$ $\ge 1$ for all $t \ge 1$ by \cite{SVV}. 
In this case, to compute $\dstab(I)$ is a cohomological problem.
If $G$ is a tree, Herzog and Hibi \cite{HH2} gave an upper bound for $\dstab(I)$, which later turned out to be the exact value of $\dstab(I)$ \cite{TNT}.  This fact already indicates the difficulty of the above problem.
The only other case where one has an explicit formula for $\dstab(I)$ is the case of a unicyclic bipartite graph \cite{TNT}. 

The main result of this paper gives an explicit formula for $\dstab(I)$ when $G$ is an arbitrary connected bipartite graph. Combining it with the result of \cite{LT} for non-bipartite graphs we get a {\it complete solution} to the above problem for ideals generated by squarefree monomials of degree 2. 

The formula is the result of the interactions between several seemingly different topics in commutative algebra, algebraic topology, integer programming, and combinatorics. 
The starting point is the finding of \cite{TNT} that $\displaystyle \lim_{t \to \infty}\depth R/I^t = 1$ and 
$$\dstab(I) = \min\{t|\ \depth R/I^t = 1\}.$$
By \cite{MT1}, $\depth R/I^t = 1$ if and only if there exists a disconnected simplicial complex associated with $I^t$ whose facets are determined by linear inequalities.
From this it follows that $\dstab(I)$ is given by the minimum order of the solutions of a diophantine system of linear inequalities (Theorem \ref{system}). 
Our insight is that a solution of this diophantine system gives rise to a matching-covered parallelization (Theorem \ref{matching-covered}).
A parallelization of a graph is a graph obtained by deletion and duplication of the vertices \cite[Part VIII]{Sch}. It is matching-covered if every edge is contained in a perfect matching \cite{Lo2,LP}. A matching-covered graph can be characterized in terms of ear decompositions.

An ear decomposition of a connected graph is a partition of the edges into a sequence of paths $L_1,...,L_r$ such that $L_1$ is a cycle and $L_i$, $i \ge 2$, meets $L_1+\cdots+L_{i-1}$ exactly at the endpoints of $L_i$.  
The paths $L_1,...,L_r$ are called ears. 
% Ear decompositions may be used to characterize several important graph classes.
A connected graph has an ear decomposition if and only if it is 2-edge connected, i.e. it remains connected after removing any edge \cite{Ro}. 
By a celebrated result of Lovasz and Plumer \cite{LP}, a connected bipartite graph is matching-covered if and only if it has an ear decomposition with only one even ear. 

The base graph of a matching-covered parallelization needs not be 2-edge connected but has degenerate ear decompositions, which allow an edge with a specified endpoint as an ear. 
We will view such an ear as a cycle of length 2 (since it is like a closed walk of length 2) and call it a 2-cycle.
A partition of the edges into a sequence of paths or 2-cycles $L_1,...,L_r$ such that $L_1$ is a cycle  and 
$L_i$, $i \ge 2$, meets $L_1+\cdots+L_{i-1}$ exactly at the endpoints of $L_i$ is called a generalized ear decomposition. 
This notion was introduced in \cite{LT} for the study of the associated primes of $I^t$. 
It has the advantage that every connected graph has generalized ear decompositions. 
We shall see that the order of matching-covered parallelizations is related to the number of even ears in
generalized ear decompositions of the base graph (Theorem \ref{min}).
 
If $G$ is a connected graph, we denote  by $\varphi(G)$ the minimal number of even ears in generalized ear decompositions of $G$ and set
$$\mu(G) := (\varphi(G)+n-1)/2.$$
This notion has its origin in coding theory \cite{SZ}.

For every subset $U \subseteq V$, we denote by $G_U$ the induced subgraph of $G$ on $U$.
We call $G_U$ dominating if $U$ is a dominating set of $G$, i.e. every vertex of $V\setminus U$ is adjacent to a vertex of $U$.
If $G$ is not a complete bipartite graph, we define
$$s(G) := \min\{\mu(G_U)|\ G_U \text{ is a dominating connected induced subgraph of $G$}\}.$$
If $G$ is a complete bipartite graph, we set $s(G) = 0$.
We can characterize $\dstab(I)$ in terms of $s(G)$ as follows. 
\medskip

\noindent {\bf Theorem \ref{main}}. 
Let $I$ be the edge ideal of a connected bipartite graph $G$. Then 
$$\dstab(I) = s(G) + 1.$$

There are many interesting applications of the above formula.
First, we can completely describe all connected bipartite graphs with $\dstab(I) = 1, 2, 3, 4$ (Corollary \ref{small}).
Such a result seems to be out of reach before.
We can easily deduce or improve all previous results on $\dstab(I)$ for connected bipartite graphs. 
For instances, we give a sharp bound for $\dstab(I)$ in terms of disjoint cycles of $G$ (Theorem \ref{bound}), from which we immediately obtain the closed formulas for $\dstab(I)$ in the case $G$ is a tree or a unicyclic graph in \cite{TNT}.
We can easily prove the bound $\dstab(I) \le n/2$ for well-covered graphs, which is a main result of \cite{BM}. 
Moreover, we can characterize those graphs for which this bound is attained (Theorem \ref{well-covered}). 
These applications are meant to demonstrate the effectiveness of the above formula.

Combining Theorem \ref{main} with the result in the non-bipartite case \cite{LT} we obtain a general formula for an arbitrary graph. For that we need some additional notions.

Let $G_1,...,G_c$ be the connected components of $G$.
We call $G$ {\it strongly nonbipartite} if $G_i$ is nonbiparite for all $i = 1,...,c$.
A generalized ear decomposition of $G$ is a sequence of generalized ear decompositions of $G_1,...,G_c$.
If $G$ a strongly nonbipartite graph, we denote by $\varphi^*(G)$ the minimum number of even ears in generalized ear decompositions of $G$ which starts with an odd cycle in all $G_i$, $i = 1,...,c$, and set
$$\mu^*(G) := (\varphi^*(G)+n-c)/2.$$

If $G$ a connected nonbipartite graph, we denote by $s(G)$ the minimum number of $\mu^*(G_U)$ among dominating strongly nonbipartite induced subgraphs $G_U$ of $G$. This notion has been used in \cite{LT} to study 
the stability of the set of associated primes of $I^t$, from which it follows that $\dstab(I) = s(G)+1$.  
\medskip

\noindent {\bf Theorem \ref{arbitrary}}.
Let $G$ be an arbitrary graph without isolated vertices. Let $G_1,...,G_c$ be the connected components of $G$. Then
$$\dstab(I) = s(G_1) + \cdots + s(G_c)+1.$$

The approach in the bipartite case differs from the nonbipartite case in both algebraic and combinatorial aspects. That is  the reason why the definition of $s(G)$ is different in these cases. 

The proof of Theorem \ref{main} is almost combinatorial in nature. It consist of several steps passing the original problem
first to a problem of integer programing, then to graph parallelizations and finally to ear decompositions. During the proof we have to solve combinatorial problems which has not been addressed before, namely,
\begin{itemize}
\item the characterization of the solutions of a class of diophantine systems of linear equations by means of matching-covered parallelizations (Theorem \ref{matching-covered}),
\item the estimate of the minimal order of matching-covered parallelizations by means of ear decompositions (Theorem \ref{min}),
\end{itemize}
The solutions of these problems are not merely applications of known results in combinatorics but requires new concepts such that generalized ear decompositions. This paper together with \cite{LT} shows that graph parallelizations and ear decompositions are effective tools in the study of cohomological behaviors of powers of edge ideals. 
Graph parallelizations have been used to study other properties of edge ideals in \cite{DRV, MMV, MRV,Tr}.
\par

The paper is divided into 4 sections. In Section 1 we present the link between $\dstab(I)$ and the feasibility of a diophantine system of linear equations and inequalities. In Section 2 we characterize the solutions of this system by means of matching-covered parallelizations. In Section 3 we describe the minimal order of matching-covered parallelizations in terms of generalized ear decompositions of the base graph. In Section 4 and Section 5 we prove general formulas for $\dstab(I)$ in the case $G$ is a connected bipartite graph or an arbitrary graph, respectively. \par

We assume that the reader is familiar with basic concepts of graph theory. 
% For convenience, we denote an edge $\{v_1,v_2\}$ by $v_1,v_2$ and a path on the vertices $v_1,...,v_r$ such that $v_iv_{i+1}$ are edges of the path, $i = 1,...,r-1$, as an ordered sequence $v_1,...,v_r$. We do not require that $v_1,v_r$ are distinct. Hence, a cycle is a path with $v_r = v_1$.
For unexplained notions in commutative algebra we refer the reader to the book of Eisenbud \cite{Ei}.

%%%%%%%%%%%%%%%%%%%%%%

\section{Depth stability and diophantine systems of linear inequalities}

This section is a preliminary to the stability of the function $\depth R/I^t$ when $I$ is the edge ideal of a simple graph $G$. We shall see that if $G$ is a connected bipartite graph, $\dstab(I)$ can be expressed in terms of a diophantine systems of linear inequalities.

By the following result of \cite{TNT},  $\displaystyle \lim_{t \to \infty} \depth R/I^t$ is known and $\dstab(I)$ is the place where $\depth R/I^t$ reaches this limit.

\begin{Theorem} \label{general} \cite[Theorem 4.4]{TNT}
Let $G$ be a simple graph without isolated vertices.  
Let $s$ be the number of the connected bipartite components of $G$. Then \par
{\rm (1)} $\min\{\depth R/I^t|\ t \ge 1\} = s$,\par
{\rm (2)} $\dstab(I) = \min\{t|\ \depth R/I^t = s\}$, \par
{\rm (3)} $\dstab(I) = \sum_{j =1}^c\dstab(I_j) - c+1,$ where $I_1,...,I_c$ are the edge ideals of the connected components of $G$.
\end{Theorem}

In particular, the computation of $\dstab(I)$ can be reduced to the case $G$ is a connected graph.
If $G$ is a connected nonbipartite graph, $\displaystyle \lim_{t \to \infty}\depth R/I^t = 0$
and there is an explicit formula for $\dstab(I)$ in terms of $G$ \cite[Corollary 3.8]{LT}.
If $G$ is a connected bipartite graph, $\displaystyle \lim_{t \to \infty}\depth R/I^t = 1$ and no explicit formula for $\dstab(I)$ was known. However, there is a criterion for $\depth R/I^t = 1$ in terms of the so called degree complex, which is the starting point of our investigation. \par

Let $J$ be an arbitrary monomial ideal in $R$.
It is well known that 
$$\depth R/J = \min\{i|\ H_\mm^i(R/J) \neq 0\},$$ where $H_\mm^i(R/J)$ denotes the $i$--th local cohomology module of $R/J$ with respect to the maximal homogeneous ideal $\mm$ of $R$.
Since $R/J$ is an $\NN^n$-graded ring, $H_\mm^i(R/J)$ is $\ZZ^n$-graded. 
For $\a \in \ZZ^n$, the $\a$--components of $H_\mm^i(R/J)$ can be expressed in terms of the reduced cohomology of a simplicial complex $\D_a(J)$ \cite{Ta}. We call it a {\it degree complex} of $J$. This concept was developed in order to generalize Hochster's formula for squarefree monomial ideals \cite{Ho}. 
If $J$ is an intersection of powers of prime ideal, the facets of $\D_a(J)$ can be described in terms of a system of linear inequalities. We refer the reader to \cite{MT1,MT} for a detailed presentation. 

% Let $J$ be a monomial ideal. Then $R/J$ has a natural $\NN^n$-graded structure. For every degree $\a = (a_1,...,a_n) \in \NN^n$, let $x^\a = x_1^{a_1}\cdots x_1^{a_n}$ and $$\D_\a(J)  := \{F \subseteq V|\ x^\a \not\in JR[x_i^{-1}|\ i \in F]\}.$$ Then $\D_\a(J)$ is a simplicial complex, which we call a {\em degree complex} of $I$. By a result of Takayama \cite{Ta}, the local cohomology modules of $R/J$ in degree $\a$ depends on the reduced cohomology of $\D_\a(I)$. The above definition of $\D_\a(I)$ is due to \cite[Lemma 1.2]{MT}, which is simpler than the original definition. \par

% Let $\D(I)$ denote the simplicial complex such that $\sqrt{I}$ is the Stanley-Reisner ideal of $\D(I)$, i.e. $\sqrt{I}$ isgenerated by the squarefree monomials whose supports correspond to the non-faces of $\D(I)$. If $I$ be the edge ideal of a bipartite graph $G$.

Let $I^{(t)}$ denote the $t$--th symbolic power of $I$, which is the intersection of the $t$-th powers of the minimal primes 
of $I$. By a pioneering work of Simis, Vasconcelos and Villarreal \cite{SVV}, 
we know that $I^t = I^{(t)}$ for all $t \ge1$ if $G$ is a bipartite graph. 
Therefore, the degree complexes of $I^t$ can be described in terms of systems of linear inequalities. \par

Let $\Max(G)$ denote the set of maximal independent sets of $G$.

\begin{Lemma} \label{facet} \cite[Lemma 1.5]{TNT}
Let $G$ be a bipartite graph. For all $\a \in \NN^n$,
$$\D_\a(I^t) = \big\langle F \in  \Max(G)|\  \sum_{i \not\in F}a_i  < t \big\rangle.$$
\end{Lemma}

The angle brackets means that the sets inside are the facets of the simplicial complex.

\begin{Example} For $t = 1$ we have $\D_0(I) = \big\langle \Max(G)\big\rangle$ and $\D_\a(I) = \emptyset$ for $\a = (1,...,1)$. \end{Example}

There is the following criterion for $\depth R/I^t = 1$ in terms of degree complex.

\begin{Lemma} \label{TNT} \cite[Lemma 3.1]{TNT}
Let $G$ be a connected bipartite graph with bipartition $(A,B)$. 
Then $\depth R/I^t = 1$ if and only if there exists $\a \in \NN^n$ such that
$\D_\a(I^t) = \langle A,B \rangle$. Moreover, if $t = \dstab(I)$, then such $\a$ satisfies the condition
$$\sum_{i \not\in A}a_i  = \sum_{i \not\in B}a_i = t-1.$$
\end{Lemma}

By Lemma \ref{facet} the condition $\D_\a(I^{t})  = \langle A,B \rangle$ means that $\a$ is a solution of the following system of linear inequalities
\begin{align*}
\sum_{i \not\in A}a_i  & < t,\
\sum_{i \not\in B}a_i  < t,\\
\sum_{i \not\in F}a_i  & \ge t,\ F \in \Max(G), F \neq A,B.
\end{align*}

Since $V = A \sqcup B$, we have
$$\sum_{i \not\in A}a_i  = \sum_{i \in B}a_i,\ \sum_{i \not\in B}a_i  = \sum_{i \in A}a_i.$$
Moreover, we can always increase the components $a_i$ to get a new solution with 
$\sum_{i \in A}a_i  = t-1$ and $\sum_{i \in B}a_i  =  t-1$. 
From this it follows that 
$$|\a| = \sum_{i \in A}a_i +  \sum_{i \in B}a_i = 2(t-1),$$
where $|\a| := a_1 + \cdots + a_n$.

Let $\Min(G)$ denote the set of minimal (vertex) covers of $G$.
Note that $F \in \Max(G)$ if and only if $V - F \in \Min(G)$.
By Lemma \ref{TNT}, $\depth R/I^t = 1$ if and only if the following diophantine system of linear equations and inequalities is feasible:
\begin{align*} 
\sum_{i \in A}a_i  & = t-1,\
\sum_{i \in B}a_i  = t-1,\\
\sum_{i \in C}a_i  & \ge t,\ C \in \Min(G), C \neq A,B.
\end{align*}

The latter system can be translated in a straightforward manner into the following system without involving the parameter $t$:
\begin{equation} \left.
\begin{array}{ll}
\displaystyle \sum_{i \in A}a_i  & =  \sum_{i \in B}a_i,\\
\displaystyle \sum_{i \in C}a_i  & > \sum_{i \in A}a_i, \ C \in \Min(G), C \neq A,B. 
\end{array}
\right. \tag{$*$} 
\end{equation}

Therefore, we obtain the following characterization for $\D_\a(I^t) = \langle A,B \rangle$.

\begin{Lemma} \label{solution}
Let $G$ be a connected bipartite graph with bipartition $(A,B)$. 
Then $\D_\a(I^t) = \langle A,B \rangle$ if $\a$ is a solutions of $(*)$ and $t = |\a|/2+1$.
Conversely, if $\D_\a(I^t) = \langle A,B \rangle$, there exits $\a' \ge \a$ such that $\a'$ is a solutions of $(*)$ and $t = |\a'|/2+1$. 
\end{Lemma}

Summing up we have the following formula for $\dstab(I)$ in terms of the solutions of a diophantine system of linear equations and inequalities.

\begin{Theorem} \label{system}
Let $G$ be a connected bipartite graph with bipartition $(A,B)$. Then 
$$\dstab(I) = \min\{ |\a|/2|\ \a \in \NN^n \text{ is a solutions of }  (*)\}+1.$$
\end{Theorem}

\begin{proof}
By Theorem \ref{general}(2), $\dstab(I) = \min\{t|\ \depth R/I^t = 1\}$.
By Lemma \ref{TNT}, $\depth R/I^t = 1$ if and only if there exist $\a \in \NN^n$ such that $\D_\a(I^t)  = \langle A,B \rangle$. Hence the conclusion follows from Lemma \ref{solution}.
\end{proof}

\begin{Remark}
By Theorem \ref{system}, we have to solve the integer programming problem that minimize the objective function $|\a|$, subject to the constraints $(*)$.
\end{Remark}

We do not solve the diophantine system $(*)$ directly. Instead, we will characterize the solutions $\a \in \NN^n$ of $(*)$ in terms of $G$. First, we observe that the case $\a = 0$ imposes a strict condition on $G$.

\begin{Lemma}  \label{a=0}
Let $G$ be a connected bipartite graph with bipartition $(A,B)$. 
Then $\a = 0$ is a solution of the system $(*)$ if and only if $G$ is a complete bipartite graph.
\end{Lemma}

\begin{proof}
It is clear that $(*)$ has the solution $\a = 0$ if and only if there is no $C \in \Min(G), C \neq A,B.$
This condition is satisfied if and only if every vertex of $A$ is adjacent to every vertex of $B$. 
That means $G$ is a complete bipartite graph. 
\end{proof}

%%%%%%%%%%%%%%%%

\section{Graph parallelization}

In this section we give a graph-theoretical characterization of the solutions $\a \neq 0$ of the system $(*)$.
For that we need the following notions.

For every $\a \in \NN^n$ we denote by $G^\a$ the graph obtained from $G$ by deleting the vertices $i$ with $a_i = 0$ and replacing every vertex $i$ with $a_i > 0$ by the vertices $i_r$, $r = 1,...,a_i$, 
where $\{i_r,j_s\}$ is an edge of $G^\a$ if $\{i,j\}$ is an edge of $G$ with $a_i > 0$, $a_j > 0$. 
Following \cite[77.3]{Sch} we call $G^\a$ the $\a$--{\it parallelization} of $G$.

The vertex $i_r$ is called a {\it duplication} of $i$. The duplication vertices are characterized by the properties that they are non adjacent to each other and have the same neighborhood. Recall that the neighborhood of a vertex $v$ is the set of all adjacent vertices, which is denoted by $N(v)$.

\begin{figure}[ht!]
\begin{tikzpicture}[scale=0.5] 

\draw [thick] (0,0) coordinate (a) -- (2,0) coordinate (b) ;
\draw [thick] (2,0) coordinate (b) -- (4,0) coordinate (c) ;
\draw [thick] (4,0) coordinate (c) -- (6,0) coordinate (d) ; 

\draw [thick] (12,0) coordinate (e) -- (14,0) coordinate (f);
\draw [thick] (14,0) coordinate (f) -- (16,0) coordinate (g);
\draw [thick] (16,0) coordinate (g) -- (18,0) coordinate (h);

\draw [thick] (12,0) coordinate (i) -- (14,2) coordinate (i);
\draw [thick] (14,0) coordinate (f) -- (16,2) coordinate (j);
\draw [thick] (16,0) coordinate (g) -- (14,2) coordinate (i);
\draw [thick] (14,2) coordinate (i) -- (16,2) coordinate (j);
\draw [thick] (18,0) coordinate (h) -- (16,2) coordinate (j);

\fill (a) circle (4pt);
\fill (b) circle (4pt);
\fill (c) circle (4pt);
\fill (d) circle (4pt);
\fill (e) circle (4pt);
\fill (f) circle (4pt);
\fill (g) circle (4pt);
\fill (h) circle (4pt);
\fill (i) circle (4pt);
\fill (j) circle (4pt);

\draw (0,-1) node{1};
\draw (2,-1) node{2};
\draw (4,-1) node{3};
\draw (6,-1) node{4};
\draw (12,-1) node{$1_1$};
\draw (14,-1) node{$2_1$};
\draw (16,-1) node{$3_1$};
\draw (18,-1) node{$4_1$};
\draw (14,3) node{$2_2$};
\draw (16,3) node{$3_2$};

\draw (3,-3) node{$G$};
\draw (15,-3) node{$G^\a$};
\draw (9.3,-3) node{$\a = (1,2,2,1)$};
 
\end{tikzpicture}
\caption{}
\end{figure}

\begin{Example} \label{parallel}
Let $G$ be the path $1,2,3,4$ and $\a = (1,2,2,1)$. Then $G^\a$ is a graph on 6 vertices and 8 edges as in  Figure 1. 
\end{Example}

Let $V^\a$ denote the vertex set of $G^\a$. 
For every vertex $i_r \in V^\a$ we set $p(i_r) = i$.
This defines a map $p: V^\a \to V$, which we call the {\it projection} of $V^\a$ to $V$.

Let $\supp(\a) = \{i|\ a_i \neq 0\}$. We call the induced subgraph $G_{\supp(\a)}$ of $G$ the {\it base graph} of $G^\a$. 
Actually, $G^\a$ is a parallelization of $G_{\supp(\a)}$. 
There is an one-to-one correspondence between their minimal covers. 

\begin{Lemma} \label{cover}
Let $U = \supp(\a)$. 
A subset $Z \subseteq V^\a$ is a minimal cover of $G^\a$ if and only if $Z = p^{-1}(D)$ for some minimal cover $D$ of $G_U$.
\end{Lemma}

\begin{proof}
Let $Z$ be a minimal cover of $G^\a$ and $D = p(Z)$.
If $Z \neq p^{-1}(D)$, there are vertices $v \in Z$, $w \not\in Z$ such that $p(v) = p(w)$.
Since $Z$ is a cover of $G^\a$, $N(w) \subseteq Z$.
By the definition of a parallelization, $N(v) = N(w) \subseteq Z$.
Therefore, we can remove $v$ from $Z$ and still obtain a cover of $G^\a$, which is a contradiction to the minimality of $Z$. So we get $Z = p^{-1}(D)$. \par

Since $Z$ meets all edges of $G^\a$, $D$ meets all edges of $G_U$. 
Hence, $D$ is a cover of $G_U$.
If $D$ is not a minimal cover of $G_U$, there is a vertex $v \in D$ such that
$D - v$ is a cover of $G_U$. Since $Z = p^{-1}(D)$, $Z \setminus p^{-1}(v)$ is also a cover of $G^\a$, which is a contradiction to the minimality of $Z$. So we can conclude that $D$ is a minimal cover of $G$.\par 

Conversely, let $Z = p^{-1}(D)$ for some minimal cover $D$ of $G_U$.
By the definition of parallelization, $Z$ is also a cover of $G^\a$.
Let $Z' \subseteq Z$ be a minimal cover of $G^\a$.
By the first part of this proof, there is a minimal cover $D'$ of $G_U$ such that $Z' = p^{-1}(D')$.
Since $D' = p(Z') \subseteq p(Z) = D$, the minimality of $D$ implies $D' = D$, whence $Z' = Z$.
% If $Z$ is not a minimal cover of $G^\a$, there is a vertex $v \in Z$ such that $Z - v$ is a cover of $G^\a$. From this it follows $p(Z-v) = D - p(v)$ is also a cover of $G_U$, which is a contradiction to the minimality of $D$.
\end{proof}

\begin{Lemma} \label{size}
Let  $U = \supp(\a)$. For every minimal cover $Z$ of $G^\a$, there is a minimal cover $C$ of $G$ such that $Z = p^{-1}(C \cap U)$ and
$$|Z| = \sum_{i \in C}a_i.$$
\end{Lemma}

\begin{proof}
By Lemma \ref{cover}, $Z = p^{-1}(D)$ for some minimal cover $D$ of $G_U$. 
Since $U - D$ is a maximal independent set in $U$, 
we can extend $U - D$ to a maximal independent set in $V$. 
The complement of this set is a minimal cover $C$ of $G$. It is clear that $D =  C \cap U$.
Since $a_i = 0$ for $i \not\in U$, we have
$$|Z| = \sum_{i \in D}a_i = \sum_{i \in C}a_i.$$ 
\end{proof}

Let $G$ be a connected bipartite graph with bipartition $(A,B)$. Let $\a \neq 0$ be a solution of the system ($*$) in Section 1.
Applying Lemma \ref{size} to ($*$), we can see that $G^\a$ has only two minimal covers of smallest possible size. Such a minimal cover is called a {\it minimum cover}. In this case, $G^\a$ belongs to the following well-known class of graphs.

A connected graph is called {\em matching-covered}  (or 1--extendable graph) if it is not an isolated vertex and every edge is contained in a perfect matching \cite{Lo2,LP}.

\begin{Theorem} \label{LP} \cite[Theorem 4.1.1\! (ii)\! $\Leftrightarrow$\! (v)]{LP}
Let $K$ be a connected bipartite graph with a bipartition $(X,Y)$.
Then $K$ is matching-covered if and only if $K$ has exactly two minimum covers, namely $X,Y$.
\end{Theorem}

\begin{Example}
By definition, any path of length $\ge 3$ is not matching-covered. 
This can be explained by Theorem \ref{LP}. For instance, the path $1,2,3$ has only a minimum cover 1 and the path
$1,2,3,4$ has 3 minimum covers $\{1,3\}$, $\{2,3\}$, $\{2,4\}$.
\end{Example}

Using the notion of matching-covered graph we can give a combinatorial characterization of the solutions $\a \neq 0$ of the system {\rm ($*$)}.

\begin{Theorem} \label{matching-covered}
Let $G$ be a connected bipartite graph with bipartition $(A,B)$. 
Then $\a \neq 0$ is a solution of the system {\rm ($*$)} if and only if the following conditions are satisfied: \par
{\rm (1)} $\supp(\a)$ is a dominating set of $G$,\par
{\rm (2)} $G^\a$ is a matching-covered graph.
\end{Theorem}

\begin{proof}
Let $\a \neq 0$ be a solution of the system {\rm ($*$)} and $U = \supp(\a)$. Note that $a_i = 0$ for $i \not\in U$. \par

Assume that $U$ is not a dominating set of $G$. 
Then there exists a vertex $v \not\in U$ that is not adjacent to any vertex in $U$.
This means $N(v) \cap U = \emptyset$. Hence $a_i = 0$ for $i \in N(v)$.
Without restriction, we may assume that $v \in A$. Then $N(v) \subseteq B$.
Since $A$ is a cover of $G$, $(A - v) \cup N(v)$ is a cover of $G$.
Hence, $(A - v) \cup N(v)$ contains a a minimal cover $C \neq A$.
We have 
$$\sum_{i \in C} a_i \le  \sum_{i \in (A - v) \cup N(v)}a_i = \sum_{i \in A-v}a_i  \le \sum_{i \in A}a_i.$$
By $(*)$, we must have $C = B$. 
Since $C \cap (A-v) = B \cap (A-v)  = \emptyset$ and $C \subset (A-v) \cup N(v)$, we get $B = C \subseteq N(v)$. Hence $B = N(v)$. 
Therefore,
$$\sum_{i \in A}a_i = \sum_{i \in B}a_i = \sum_{i \in N(v)}a_i =0.$$
So we get $\a = 0$, a contradiction. This proves (1). \par

To prove (2) we observe that $\a \neq 0$ implies $\sum_{i\in A}a_i = \sum_{i\in B}a_i  \neq 0$.
From this it follows that $A \cap U \neq \emptyset$ and $B \cap U \neq \emptyset$. 
Therefore, $(A \cap U, B\cap U)$ is a bipartition of $G_U$. 
Let $X = p^{-1}(A \cap U)$ and $Y = p^{-1}(B \cap U)$. 
Then $(X,Y)$ is a bipartition of $G^\a$.
By Theorem \ref{LP}, it suffices to show that $G^\a$ has exactly two minimum covers $X,Y$.
Let $Z$ be an arbitrary minimal cover $G^\a$.  
By Lemma \ref{size}, there exists a minimal cover $C$ of $G$ such that $Z = p^{-1}(C \cap U)$ and
 $$|Z| = \sum_{i \in C \cap U}a_i = \sum_{i \in C}a_i.$$
 If $Z \neq X,Y$, then $C \neq A,B$. By ($*$) we have
 $$|X| = \sum_{i \in A}a_i = \sum_{i \in B}a_i = |Y| < \sum_{i \in C}a_i = |Z|.$$
 Therefore, $X$ and $Y$ are exactly the minimum covers of $G^\a$. 

Now we are going to prove the converse of Theorem \ref{matching-covered}.
Assume that the conditions (1) and (2) are satisfied. 
Let $X = p^{-1}(A \cap U)$ and $Y = p^{-1}(B \cap U)$.
Then $G^\a$ must be a connected bipartite graph with bipartition $(X,Y)$.
By (2) and Theorem \ref{LP}, $X,Y$ are the only minimum covers of $G^\a$. 
Therefore,
$$|X| = |Y| < |Z| \text{ for all } Z \in \Min(G^\a),\ Z \neq X,Y.$$
From the equation $|X| = |Y|$ we get
$$\sum_{i \in A}a_i =  \sum_{i \in B}a_i.$$

Let $C \in \Min(G), C \neq A, B$.
If $C \cap U \subseteq  A \cap U$, $(B \cap U) \cap C = \emptyset.$
Since $C$ is a minimal cover, every vertex adjacent to $B \cap U$ belongs to $C$.
From this it follows that $A-U \subseteq C$ because every vertex of $A - U$ is adjacent to $B \cap U$ by (1). 
Therefore, $A = (A -U) \cup (A \cap U) \subseteq C$, which implies $A = C$, a contradiction.
So $C \cap U \not\subseteq A \cap U$. Hence $B \cap (C \cap U) \neq \emptyset$, which implies
$\displaystyle \sum_{i \in B \cap C}a_i \ge 1$. 

Since $C \cap U$ is a cover of $G_U$, we can find a minimal cover $D \subseteq C \cap U$ of $G_U$.
Let $Z = p^{-1}(D)$.
Then $Z$ is a minimal cover of $G^\a$ by Lemma \ref{cover}. 
If $Z = X$, then $D = A \cap U$. Therefore,
$$\sum_{i \in C}a_i = \sum_{i \in A \cap C}a_i + \sum_{i \in B \cap C}a_i > \sum_{i \in A \cap D}a_i = \sum_{i \in A \cap U}a_i = \sum_{i \in A}a_i.$$
Similarly, if $Z = Y$, we also have 
$$\sum_{i \in C}a_i > \sum_{i \in B}a_i = \sum_{i \in A}a_i.$$
If $Z \neq X,Y$, then
$$\sum_{i \in C}a_i \ge \sum_{i \in D}a_i = |Z| > |X| = \sum_{i \in A}a_i.$$
Therefore, $\a$ is a solution of $(*)$. The proof of Theorem \ref{matching-covered} is now complete.
\end{proof}

The base graph of a matching-covered parallelization can be described as follows. 
For convenience, we say that $G_U$ is a dominating subgraph of $G$ if $U$ is a dominating set of $G$.

\begin{Corollary} \label{base}
Let $\a \neq 0$ be a solution of the system $(*)$ and $U = \supp(\a)$. 
Then $G_U$ is a dominating connected bipartite subgraph of $G$.
\end{Corollary}

\begin{proof}
By Theorem \ref{matching-covered}(1), $G_U$ is a dominating subgraph of $G$.
By Theorem \ref{matching-covered}(2), $G^\a$ is a matching-covered graph. 
Hence, it is not an isolated vertex and connected.
Since $G^\a$ is a parallelization of $G_U$, $G_U$ is also not an isolated vertex and connected.
Since $G$ is bipartite,  such a connected induced graph is bipartite.  
\end{proof}

Using Theorem \ref{matching-covered} and Corollary \ref{base}, we can easily find some solutions of $(*)$.

\begin{Example} \label{solution-ex}
Let $G$ be the path $1,2,3,4$. For $\a = (1,2,2,1)$, $G^\a$ has exactly two minimum covers $\{1_1,3_1,3_2\}$ and $\{2_1,2_2,4_1\}$ (see Figure 1). Hence $G^\a$ is matching-covered by Theorem \ref{LP}. Since $V = \supp(\a)$, $(1,2,2,1)$ is a solution of $(*)$ by Theorem \ref{matching-covered}. However, $|\a| = 6$ is not the minimal order of the solutions of $(*)$. For $\a = (0,1,1,0)$ we have $U = \{2,3\}$, which is a dominating set of $G$, and $G_U$ is the edge $\{2,3\}$, which is matching-covered. Hence, $(0,1,1,0)$ is also a solution of $(*)$ with $|\a| = 2$.
\end{Example}

%%%%%%%%%%%%%%%%%%%%%

\section{Generalized ear decomposition}

In this section we want to estimate the order $|\a|$ of a matching-covered parallelization $G^\a$ in terms of $G$.
For this we need a characterization of matching-covered bipartite graph in terms of ear decomposition. \par

An {\it ear decomposition} of a connected graph is a partition of the edges into a sequence of paths $L_1,...,L_r$ such that $L_1$ is a cycle  and $L_i$, $i \ge2$, meets $L_1+\cdots+L_{i-1}$ exactly at the endpoints of $L_i$. Note that an open path has two endpoints and a cycle is a closed path whose endpoints coincide.
The paths $L_1,...,L_r$ are called {\it ears} of the decomposition. An ear is called even or odd if it has even or odd length, respectively. For convenience, we denote the partition by $L_1+\cdots+L_r$. 
A connected graph has an ear decomposition if and only if it is {\it 2-edge connected}, i.e. it remains connected after removing any edge \cite{Ro}. 

\begin{Theorem} \label{one even ear} \cite[Theorem 4.1.6]{LP}
A connected bipartite graph is matching-covered if and only if it has an ear decomposition with only one even ear.
\end{Theorem}

The original result \cite[Theorem 4.1.6]{LP} uses a different definition of ear decomposition which requires $L_1$ to be an edge and $L_2$ meets $L_1$ in two endpoints. It says that a connected bipartite graph is matching-covered if and only if there is such an ear decomposition without even ears. In that case, $L_1+L_2$ is an even cycle, which together with the other ears form an ear decomposition with only one even ear. \par

If a parallelization $G^\a$ has an ear decomposition, we will lower the weight $\a$ step by step and see how the ear decomposition degenerates. 
For any $i \in \supp(\a)$ with $a_i \ge2$, $G^{\a-\e_i}$ can be obtained from $G^\a$ by deleting the vertex $i_{a_i}$.
Let $L_1+\cdots+L_r$ be an ear decomposition of $G^\a$.
Let $L_t'$ be the remainder of $L_t$ after this deletion, $t = 1,...,r$.
It may happen that a path $L_t'$ is an edge $\{u,v\}$ meeting $L_1'+\cdots +L_{t-1}'$ at only one endpoint $u$. In this case, $L_1'+\cdots +L_t'$ is not an ear decomposition of $G^{\a-\e_i}$.
So there is a need to modify the definition of ear decompositions in order to have a common framework for the procedure of deletions. For instance, we may view $L'_t$ as a closed walk $\{u,v,u\}$. Then $L'_t$ has only an endpoint $u$, and the condition $L'_t$ meets $L_1'+\cdots +L_{t-1}'$ only at the endpoints is satisfied.

\begin{Example}
Let $G^\a$ be as in Example \ref{parallel} (see Figure 2). It has an ear decomposition $L_1+L_2+L_3$, where $L_1$ is the 4-cycle $\{1_1,2_1,3_1,2_2,1_1\}$, $L_2$ is the path $\{2_1,3_2,4_1,3_1\}$ and $L_3$ is the edge $\{2_2,3_2\}$. This ear decomposition has only an even ear $L_1$. If we delete the vertex $3_2$, then $L_1' = L_1$, $L_2' = \{3_1,4_1\}$, $L_3' = \emptyset$. Since the endpoint $4_1$ of $L_2'$ does not belong to $L_1'$, 
$L_1'+L_2'$ is not an ear decomposition of $G^{\a-\e_3}$. 
\end{Example}

\begin{figure}[ht!]
\begin{tikzpicture}[scale=0.5] 

\draw [thick] (0,0) coordinate (a) -- (2,0) coordinate (b) ;
\draw [thick] (2,0) coordinate (b) -- (4,0) coordinate (c) ;
\draw [thick] (4,0) coordinate (c) -- (6,0) coordinate (d) ; 
\draw [thick] (2,0) coordinate (b) -- (4,2) coordinate (j);
\draw [thick] (4,0) coordinate (c) -- (2,2) coordinate (k);
\draw [thick] (2,2) coordinate (k) -- (4,2) coordinate (j);
\draw [thick] (0,0) coordinate (a) -- (2,2) coordinate (k);
\draw [thick] (6,0) coordinate (d) -- (4,2) coordinate (j);

\draw [thick] (12,0) coordinate (e) -- (14,0) coordinate (f);
\draw [thick] (14,0) coordinate (f) -- (16,0) coordinate (g);
\draw [thick] (16,0) coordinate (g) -- (18,0) coordinate (h);
\draw [thick] (12,0) coordinate (i) -- (14,2) coordinate (i);
\draw [thick] (16,0) coordinate (g) -- (14,2) coordinate (i);

\fill (a) circle (4pt);
\fill (b) circle (4pt);
\fill (c) circle (4pt);
\fill (d) circle (4pt);
\fill (e) circle (4pt);
\fill (f) circle (4pt);
\fill (g) circle (4pt);
\fill (h) circle (4pt);
\fill (i) circle (4pt);
\fill (j) circle (4pt);
\fill (k) circle (4pt);

\draw (0,-1) node{$1_1$};
\draw (2,-1) node{$2_1$};
\draw (4,-1) node{$3_1$};
\draw (6,-1) node{$4_1$};
\draw (2,3) node{$2_2$};
\draw (4,3) node{$3_2$};

\draw (12,-1) node{$1_1$};
\draw (14,-1) node{$2_1$};
\draw (16,-1) node{$3_1$};
\draw (18,-1) node{$4_1$};
\draw (14,3) node{$2_2$};
\draw (4,3) node{$3_2$};

\draw (3,-3) node{$G^\a$};
\draw (15,-3) node{$G^{\a-\e_3}$};
 
\end{tikzpicture}
\caption{}
\end{figure}

For convenience, we call a closed walk $\{u,v,u\}$ a {\it 2-cycle}, which has only an endpoint $u$. Depending on our choice, the edge $\{u,v\}$ can be considered as an 1-path or a 2-cycle with a specified endpoint. 

A {\it generalized ear decomposition} of a connected graph is a partition of the edges into a sequence of paths or 2-cycles $L_1,...,L_r$  such that $L_1$ is a cycle and $L_i$, $i \ge 2$, meets $L_1+\cdots+L_{i-1}$ exactly at the endpoints of $L_i$. 
In particular, $L_1$ may be a 2-cycle.

\begin{Remark} \label{existence}
The new notion has the advantage that any connected graph $G$ has at least a generalized ear decomposition.
To see that let $T$ be a spanning tree of $G$, where {\it spanning} means $T$ and $G$ has the same set of vertices. 
Order the edges of $T$ in such a way that every edge is incident to the union of the preceding edges. Then the ordered edges form a generalized ear decomposition of $G_T$, which consists of only 2-cycles. 
Since the remaining edges meet $G_T$ at both ends,
we can add them as paths of length 1 to such a decomposition to obtain a generalized ear decomposition of $G$. 
Similarly, we can show that any cycle of $G$ can be the first ear of a generalized ear decomposition whose subsequent ears are either 2-cycles or paths of length 1.
\end{Remark}

\begin{Example}
The graph $G^\a-\e_3$ of Figure 2 doesn't have any ear decomposition because it has a leaf edge. In fact, a leaf edge is not contained in any ear if we don't allow 2-cycles as ears (the leaf can not belong to any earlier ear). 
One can immediately see that $G^\a-\e_3$ has a generalized ear decompositions consisting of the 4-cycle and the leaf edge as a 2-cycle. If we remove the edge $\{2_2,3_1\}$, we have 
a spanning tree of $G^\a-\e_3$, e.g. the tree consisting of the path $1_1,2_1,3_1,4_1$ and the edge $\{1_1,2_2\}$. This tree has several generalized ear decompositions, e.g. the sequence of the 2-cycles $\{1_1,2_1\},\{2_1,3_1\},\{3_1,4_1\},\{1_1,2_2\}$. Adding the edge $\{2_2,3_1\}$ as an 1-path to this sequence we obtain a generalized ear decomposition of $G^\a-\e_3$.
\end{Example}

\begin{Remark} \label{properties}
We may represent an ear as an ordered sequence of vertices $v_1,...,v_r$ such that $v_i,v_{i+1}$ is an edge of the ear. The vertices $v_1,v_r$ are called the endpoints and $v_2,...,v_{r-1}$ the inner points of the ear.
A 2-cycle $\{v_1,v_2\}$ with endpoint $v_1$ is the sequence $v_1,v_2,v_1$, while an edge $\{v_1,v_2\}$ with two endpoints is the sequence $v_1,v_2$. 
From the definition of a generalized ear decomposition we can easily see that it has the following properties:
\begin{enumerate}
\item For every vertex $i$ different from the endpoints of the first ear, the earliest ear containing $i$ is the unique ear containing $i$ as an inner point.
\item If an ear contains an edge whose vertices belong to previous ears, then this ear consists of only this edge.
\end{enumerate}
\end{Remark}

Let $\varphi(\E)$ denote the number of even ears in a generalized ear decomposition $\E$.
The following result shows that a generalized ear decomposition of a parallelization always degenerates to a generalized ear decomposition of the deletion of a duplicated vertex such that the number of even ears increases at most by one. 

\begin{Lemma}  \label{merge} 
Let $K$ be a connected graph which has two nonadjacent vertices $i,j$ with the same neighborhood.
Let $\E$ be a generalized ear decomposition of $K$. 
Then $\E$ can be deformed to a generalized ear decomposition $\F$ of $K-j$ with $\varphi(\F) \le \varphi(\E)+1$.
\end{Lemma}

\begin{proof} 
Let $L_1, \cdots,L_r$ be the sequence of ears of $\E$. 
Let $\E_t$ denote the sequence $L_1, \cdots,L_t$, which is a generalized ear decomposition of the subgraph $L_1+\cdots+L_t$, $t = 1,...,r$.
We will show by induction that for every $t$, there exists $t' \ge t$ such that
$\E_{t'}$ can be deformed to a generalized ear decomposition $\F_{t'}$ of $L_1+\cdots+L_{t'}-j$
with $\varphi(\F_{t'}) \le \varphi(\E_{t'})+1$. Without loss of generality, we may assume that $j \in L_t$. 

Let $t$ be the least integer such that $j \in L_t$. Since the deletion of $i$ is the same as the deletion of $j$, we may further assume that $i$ belongs to $L_1+\cdots+L_t$. 
By Remark \ref{properties}(1), $j$ is an inner point of $L_t$. Let $u, v$ be the two vertices of $L_t$ adjacent to $j$. Then $u, v$ are also adjacent to $i$. We will distinguish several cases depending on the positions of $i, u,v$. \smallskip

Case 1: {\em $i$ is not an inner point of  $L_t$}. Then $i$ belongs to an earlier ear than $L_t$.

If $u,v$ are the endpoints of $L_t$, then $L_t$ is the path $u,j,v$. Set $\F_t = \E_{t-1}$. Then $\F_t$ is a generalized ear decomposition of $L_1+\cdots+L_t-j$ with  $\varphi(\F_t) = \varphi(\E_t)-1$.  

If only $u$ or $v$ is an endpoint of $L_t$, we assume that $u$ is an endpoint of $L_t$ and $v$ is an inner point of $L_t$. Then $L_t$ is the first ear containing $v$ by Remark \ref{properties}(1). 
If $\{i,v\}$ is an edge of $L_t$, then $L_t$ must be the path $\{u,j,v,i\}$ because $i$ is not an inner point of $L_t$. Adding the 2-cycle $\{i,v\}$ to $L_1,...,L_{t-1}$ we obtain a generalized ear decomposition $\F_t$ of $L_1 + \cdots + L_t-j$ with $\varphi(\F_t) \le \varphi(\E_t)+1$.
If $\{i,v\}$ is not an edge of $L_t$, then $\{i,v\}$ must belong to a late ear than $L_t$. That ear must be the edge $\{i,v\}$ by Remark \ref{properties}(2). Without restriction we may assume that $L_{t+1}$ is the edge $\{i,v\}$. Removing $j$ and adding the edge $\{i,v\}$ we obtain from $L_t$ a new path $L_t'$ starting from $i$ (see Figure 3). Note that $i$ belongs to $L_1+\cdots+L_{t-1}$.
Adding this path to $L_1,...,L_{t-1}$ we obtain a generalized ear decomposition $\F_{t+1}$ of $L_1 + \cdots + L_{t+1}-j$ with $\varphi(\F_{t+1}) \le \varphi(\E_{t+1})+1$.

\begin{figure}[ht!]
\begin{tikzpicture}[scale=0.45]

\draw [thick] (0,0) coordinate (a) -- (2,2) coordinate (b) ;
\draw [thick] (2,2) coordinate (b) -- (4,0) coordinate (c) ;
\draw [thick] (4,0) coordinate (c) -- (7,0) coordinate (d) ;
\draw [thick, dashed] (4,0) coordinate (c) -- (2,-2) coordinate (e) ;

\draw [thick, dashed] (12,0) coordinate (g) -- (14,2) coordinate (h) ;
\draw [thick, dashed] (14,2) coordinate (h) -- (16,0) coordinate (i) ;
\draw [thick] (16,0) coordinate (i) -- (19,0) coordinate (j) ;
\draw [thick] (16,0) coordinate (i) -- (14,-2) coordinate (k) ;

\fill (a) circle (4pt);
\fill (b) circle (4pt);
\fill (c) circle (4pt);
\fill (e) circle (4pt);

\fill (g) circle (4pt);
\fill (h) circle (4pt);
\fill (i) circle (4pt);
\fill (k) circle (4pt);

\draw (2,-2.7) node{$i$};
\draw (2,2.7) node{$j$};
\draw (-0.4,0.7) node{$u$};
\draw (4.2,0.7) node{$v$};

\draw (14,-2.7) node{$i$};
\draw (14,2.7) node{$j$};
\draw (11.6,0.7) node{$u$};
\draw (16.2,0.7) node{$v$};

\draw (2,1) node{$L_t$};
\draw (4.2,-1.2) node{$L_{t+1}$};
\draw (16.6,-1) node{$L_t'$};

\end{tikzpicture}
\caption{}
\end{figure}

If $u,v$ are inner points of $L_t$, then $L_t$ is the first ear containing $u$  and $v$ by Remark \ref{properties}(1).
Hence $\{i,u\}$ and $\{i,v\}$ are ears of $\E$ by Remark \ref{properties}(2). 
Without loss of generality we may assume that $L_{t+1}$ and $L_{t+2}$ are the edges $\{i,u\}$ and $\{i,v\}$. 
Removing $j$ and adding the edges $\{i,u\}$ and $\{i,v\}$ we obtain from $L_t$ a path $L_t'$ passing through $i$ (see Figure 4).
Divide this path into two paths meeting only at $i$. 
If $L_t$ is an even ear, these paths are both odd or even. If $L_t$ is odd, only one of these paths is even. 
Adding these paths to $L_1,...,L_{t-1}$ we obtain a generalized ear decomposition $\F_{t+2}$ of $L_1 + \cdots + L_{t+2}-j$ with $\varphi(\F_{t+2}) \le \varphi(\E_{t+2})+1$. \smallskip

\begin{figure}[ht!]
\begin{tikzpicture}[scale=0.45] 

\draw [thick] (0,0) coordinate (g) -- (-3,0);
\draw [thick] (0,0) coordinate (g) -- (2,2) coordinate (h) ;
\draw [thick] (2,2) coordinate (h) -- (4,0) coordinate (i) ;
\draw [thick] (4,0) coordinate (i) -- (7,0) coordinate (j) ;
\draw [thick, dashed] (4,0) coordinate (i) -- (2,-2) coordinate (k) ;
\draw [thick, dashed] (0,0) coordinate (g) -- (2,-2);

\draw [thick, dashed] (13,0) coordinate (a) -- (15,2) coordinate (b) ;
\draw [thick, dashed] (15,2) coordinate (b) -- (17,0) coordinate (c) ;
\draw [thick] (17,0) coordinate (c) -- (20,0) coordinate (d) ;
\draw [thick] (13,0) coordinate (a) -- (10,0) coordinate (e) ;
\draw [thick] (13,0) coordinate (a) -- (15,-2) coordinate (f) ; 
\draw [thick] (17,0) coordinate (c) -- (15,-2) coordinate (f) ;

\fill (a) circle (4pt);
\fill (b) circle (4pt);
\fill (c) circle (4pt);
\fill (f) circle (4pt);
\fill (g) circle (4pt);
\fill (h) circle (4pt);
\fill (i) circle (4pt);
\fill (k) circle (4pt);

\draw (2,-2.7) node{$i$};
\draw (2,2.8) node{$j$};
\draw (-0.3,0.7) node{$u$};
\draw (4.2,0.7) node{$v$};

\draw (15,-2.7) node{$i$};
\draw (15,2.8) node{$j$};
\draw (12.6,0.7) node{$u$};
\draw (17.2,0.7) node{$v$};

\draw (2,1) node{$L_t$};
\draw (4.2,-1.1) node{$L_{t+2}$};
\draw (0.2,-1.2) node{$L_{t+1}$};
\draw (15,-0.9) node{$L_t'$};

\end{tikzpicture}
\caption{}
\end{figure}

Case 2: {\em $i$ is an inner point of $L_t$}. Then $i$ doesn't belong to any earlier ear than $L_t$.
Without restriction we may assume that $u$ appears before $j$ and $j$ before $i$ in $L_t$.

If $\{i,u\} \in L_t$ and $\{i,v\} \in L_t$, then $L_t$ must be the cycle $\{u,j,v,i,u\}$.
Adding the two 2-cycles $\{u,i\}$ and $\{i,v\}$ to $L_1,...,L_{t-1}$ we obtain a generalized ear decomposition $\F_t$ of $L_1 + \cdots + L_{t}-j$ with $\varphi(\F_{t}) = \varphi(\E_{t})+1$.

If $\{i,u\} \in L_t$ and $\{i,v\} \not\in L_t$, then $L_t$ must be a cycle starting from $u$. 
By Remark \ref{properties}(2), $\{i,v\}$ must be an ear of $\E$ which appears later than $L_t$.
We may assume that $L_{t+1}$ is this ear. Removing $j$ and adding $\{i,v\}$ we obtain from $L_t$ a union of the 2-cycle $\{u,i\}$ and a cycle passing through $v$ with endpoint $i$  (see Figure 5 I). Adding them to $L_1 + \cdots + L_{t-1}$ we obtain a generalized ear decomposition $\F_{t+1}$ of $L_1 + \cdots + L_{t+1}-j$ with $\varphi(\F_{t+1}) \le \varphi(\E_{t+1})+1$.

If $\{i,u\} \not\in L_t$ and $\{i,v\} \in L_t$, then $\{i,u\}$ must be an ear of $\E$ by Remark \ref{properties}(2). We may assume that $L_{t+1}$ is this ear. Removing $j$ and adding $\{i,u\}$ we obtain from $L_t$ a union of a path passing through $u,i$ and the 2-cycle $\{i,v\}$  (see Figure 5 II). Adding them to $L_1 + \cdots + L_{t-1}$ we obtain a generalized ear decomposition $\F_{t+1}$ of $L_1 + \cdots + L_{t+1}-j$ with $\varphi(\F_{t+1}) \le \varphi(\E_{t+1})+1$.

\begin{figure}[ht!]
\begin{tikzpicture}[scale=0.45] 

\draw [thick] (0,0) coordinate (g) -- (2,2) coordinate (h) ;
\draw [thick] (2,2) coordinate (h) -- (4,0) coordinate (i) ;
\draw [thick] (4,0) coordinate (g) -- (6,-2) coordinate (j) ;
\draw [thick] (0,0) coordinate (g) -- (2,-2) coordinate (k) ; 
\draw [thick] (2,-2) coordinate (k) -- (4,-4) coordinate (l) ; 
\draw [thick] (6,-2) coordinate (j) -- (4,-4) coordinate (l) ; 
\draw [thick, dashed] (4,0) coordinate (i) -- (2,-2) coordinate (k) ;

\draw [thick] (12,0) coordinate (a) -- (14,2) coordinate (b) ;
\draw [thick] (14,2) coordinate (b) -- (16,0) coordinate (c) ;
\draw [thick] (16,0) coordinate (c) -- (14,-2) coordinate (d) ;
\draw [thick] (12,0) coordinate (a) -- (9,0) coordinate (e) ;
\draw [thick] (14,-2) coordinate (d) -- (11,-2) coordinate (g) ;
\draw [thick, dashed] (12,0) coordinate (a) -- (14,-2) coordinate (f) ; 

\draw [thick] (24,0) coordinate (a) -- (26,2) coordinate (b) ;
\draw [thick] (26,2) coordinate (b) -- (28,0) coordinate (c) ;
\draw [thick, dashed] (28,0) coordinate (c) -- (26,-2) coordinate (d) ;
\draw [thick] (28,0) coordinate (c) -- (30,-2) coordinate (h) ;
\draw [thick] (30,-2) coordinate (h) -- (28,-4) coordinate (i) ;
\draw [thick] (26,-2) coordinate (d) -- (28,-4) coordinate (i) ;
\draw [thick] (24,0) coordinate (a) -- (21,0) coordinate (e) ;
\draw [thick] (26,-2) coordinate (d) -- (23,-2) coordinate (g) ;
\draw [thick, dashed] (24,0) coordinate (a) -- (26,-2) coordinate (f) ;

\fill (0,0) circle (4pt);
\fill (2,2) circle (4pt);
\fill (4,0) circle (4pt);
\fill (2,-2) circle (4pt);
\fill (12,0) circle (4pt);
\fill (14,2) circle (4pt);
\fill (16,0) circle (4pt);
\fill (14,-2) circle (4pt);
\fill (24,0) circle (4pt);
\fill (26,2) circle (4pt);
\fill (28,0) circle (4pt);
\fill (26,-2) circle (4pt);

\draw (2,-2.7) node{$i$};
\draw (2,2.7) node{$j$};
\draw (-0.3,0.6) node{$u$};
\draw (4.2,0.7) node{$v$};

\draw (14,-2.7) node{$i$};
\draw (14,2.7) node{$j$};
\draw (11.7,0.7) node{$u$};
\draw (16.2,0.7) node{$v$};

\draw (26,-2.7) node{$i$};
\draw (26,2.7) node{$j$};
\draw (23.7,0.7) node{$u$};
\draw (28.2,0.7) node{$v$};

\draw (2,1) node{$L_t$};
\draw (14,1) node{$L_t$};
\draw (26,1) node{$L_t$};
\draw (4.1,-1.3) node{$L_{t+1}$};
\draw (12.2,-1.3) node{$L_{t+1}$};
\draw (24.1,-1.3) node{$L_{t+1}$};
\draw (28.2,-1.3) node{$L_{t+2}$};

\draw (2,-5) node{I};
\draw (14,-5) node{II};
\draw (26,-5) node{III};

\end{tikzpicture}
\caption{}
\end{figure}

If $\{u,i\} \not\in L_t$ and $\{i,v\} \not\in L_t$, then $\{i,u\}$ and $\{i,v\}$ must be ears of $\E$ by Remark \ref{properties}(2). We may assume that $L_{t+1}$ and $L_{t+2}$ are these ears. Removing $j$ and adding the edges $\{i,u\}$, $\{i,v\}$ we obtain from $L_t$ a union of a path passing through $u,i$ and a cycle passing through $v$ with endpoint $i$  (see Figure 5 III). Adding them to $L_1 + \cdots + L_{t-1}$ we obtain a generalized ear decomposition $\F_{t+2}$ of $L_1 + \cdots + L_{t+2}-j$ with $\varphi(\F_{t+2}) \le \varphi(\E_{t+2})+1$.

Let $s > t$ be an index such that $j \in L_s$. Then $j$ is an endpoint of $L_s$ by Remark \ref{properties}(1).
Using induction we may assume that there is a generalized ear decomposition $\F_{s-1}$ of the subgraph $L_1 + \cdots + L_{s-1}-j$ with $\varphi(\F_{s-1}) \le \varphi(\E_{s-1})+1$.
Let $u$ be a vertex of $L_s$ adjacent to $j$. 

If $u$ is an endpoint of $L_s$, then $L_s$ is the edge $\{j,u\}$. 
In this case, $\varphi(\E_{s-1}) = \varphi(\E_s)$ and $\F_{s-1}$ is a generalized ear decomposition of $L_1 + \cdots + L_s-j$. Set $\F_s = \F_{s-1}$. Then $\varphi(\F_s) \le  \varphi(\E_s)+1.$

If $u$ is an inner point of $L_s$, then $u$ does not belong to the previous ears by Remark \ref{properties}(1). 
Hence the edge $\{i,u\}$ belongs a late ear than $L_s$. 
This ear must be the edge $\{i,u\}$ by Remark \ref{properties}(2). 
Without restriction we may assume that this ear is $L_{s+1}$. 
Removing $j$ and adding the edge $\{i,u\}$ to $L_s$ we obtain a new path of the same length as $L_s$.
Adding this path to $\F_{s-1}$ we obtain a generalized ear decomposition $\F_{s+1}$ of $L_1+\cdots+L_{s+1}-j$. The difference $\varphi(\F_{s+1}) - \varphi(\F_{s-1})$ is 0 or 1 if the length of this new path is odd or even. Similarly, the difference
$\varphi(\E_s) - \varphi(\E_{s-1})$ is 0 or 1 if the length of $L_s$ is odd or even. Hence
$\varphi(\F_{s+1}) - \varphi(\F_{s-1}) = \varphi(\E_s) - \varphi(\E_{s-1}).$
Therefore, 
$$\varphi(\F_{s+1}) \le  \varphi(\F_{s-1}) + \varphi(\E_s) - \varphi(\E_{s-1}) 
\le \varphi(\E_s)+1.$$

By this way, we can find a generalized ear decomposition $\F$ of $K-j$ such that $\varphi(\F) \le \varphi(\E)+1$.
\end{proof}

For a connected graph $G$ let $\varphi(G)$ denote the minimum number of even ears in generalized ear decompositions of $G$. This notion has been used to denote the minimum number of even ears in ear decompositions if $G$ is a 2-edge connected graph. The following lemma shows that the two notions are the same in this case.  

\begin{Lemma} \label{phi}
Let $G$ be a 2-edge connected graph. 
Then $\varphi(G)$ is the minimum number of even ear in ear decompositions of $G$. 
\end{Lemma}

\begin{proof}
It suffices to show that if $\E$ is a generalized ear decomposition of $G$, there exists an ear decomposition $\F$ of $G$ such that $\varphi(\F) \le \varphi(\E)$. Note that an ear decomposition of $G$ is a generalized ear decomposition without 2-cycles. Without restriction we may assume that $\E$ has 2-cycles.

Let $ L_1,...,L_r$ be the sequence of ears of $\E$.
Let $L_i$ be the last 2-cycle appearing in $\E$ and $L_i = u,v,u$.
By Remark \ref{properties}(1), $L_i$ is the unique ear containing $v$ as an inner point.
Since $G$ is 2-edge connected, $v$ is not a leaf. 
Hence, there is an other ear of $\E$ containing $v$.
Let $L_j$ be the first ear of $\E$ which contains $v$ as an endpoint. Then $j > i$.
Since the ears $L_{i+1},...,L_{j-1}$ do not contain $v$, we may change the order of $\E$ to 
$$L_1,...,L_{i-1},L_{i+1},...,L_{j-1},L_i,L_j,...,L_r$$ 
and still obtain a generalized ear decomposition of $G$. 
By Remark \ref{properties}, the inner points of $L_j$ are different than $u,v$.
Putting the edge $\{u,v\}$ and $L_j$ together we obtain a path $L_j'$, which is even or odd if $L_j$ is odd or even, respectively. Let $\F$ be the sequence
$$L_1,...,L_{i-1},L_{i+1},...,L_{j-1},L_j',L_{j+1},...,L_r.$$ 
Then $F$ is a generalized ear decomposition of $G$ with $\varphi(F) \le \varphi(G)$.
Moreover, the number of 2-edges of $\E$ is greater than that of $\F$.
If $\F$ has a 2-cycle, we can repeat this replacement again until we obtain a generalized ear decomposition $\F$ without 2-cycles with $\varphi(\F) \le \varphi(\E)$. 
\end{proof}

Applying Lemma \ref{merge}, we obtain the following bound for the order of a matching-covered parallelizations whose base graph is $G$.

\begin{Proposition}  \label{order} 
Let $G^\a$ be a matching-covered graph with $\supp(\a) = V$. Then 
$$|\a| \ge\varphi(G)+n-1.$$
\end{Proposition}

\begin{proof} 
By the definition of a parallelization, $G$ can be obtained from $G^\a$ by deleting $a_i-1$ duplications of $i$, $i = 1,...,n$. By Theorem \ref{one even ear}, there is  an ear decomposition $\E$ of $G^\a$ with $\varphi(\E) = 1$.
Applying Lemma \ref{merge}, $\E$ can be deformed to a generalized ear decomposition $\F$ of $G$ with
$$\varphi(\F) \le \varphi(\E) + \sum_{i=1}^n (a_i -1) = |\a| -n + 1.$$
Therefore,
$$|\a| \ge\varphi(\F)+n-1 \ge\varphi(G)+n-1.$$
\end{proof}

On the other hand, we can always construct a matching-covered parallelization $G^\a$ with $\supp(\a) = V$ and $|\a| = \varphi(G) + n-1.$ 

\begin{Proposition} \label{construction}
Let $G$ be a connected bipartite graph. Let $\E$ be a generalized ear decomposition of $G$. 
Let $\a \in \NN^n$ such that $a_i -1$ is the number of even ears of $\E$ which is different from the first ear and which starts with the vertex $i$, $i = 1,...,n$.
Then $G^\a$ is a matching-covered graph with
$$|\a| = \varphi(\E) + n-1.$$
\end{Proposition}

\begin{proof} 
By the definition of $\a$, we have $|\a| = \varphi(\E) + n-1$.
If $\a = (1,...,1)$, then $\varphi(\E) = 1$. Hence $G^\a = G$ is matching-covered by Theorem \ref{one even ear}.  

If $\a \neq (1,...,1)$, we may assume that $a_n > 1$.
Let  $L_1,...,L_r$ be ears of $\E$. 
Let $t$ be the least index $> 1$ such that $L_t$ is an even ear starting with the vertex $n$.

Let $K$ be the graph obtained from $G$ by adding a duplication $n+1$ of $n$ to $G$.
This means $n$ and $n+1$ are not incident and have the same neighborhood.
Let $u$ be the vertex adjacent to $n$ in $L_t$.
Let $v$ be a vertex of the subgraph $L_1+\cdots+L_{t-1}$ adjacent to $n$. 
Then $\{u,n+1\}$ and $\{v,n+1\}$ are edges of $K$.
Let $H$ be graph obtained from $G$ by adding these edges.
Let $L_t'$ be the odd ear obtained from $L_t$ by replacing the edge $\{u,n\}$ by the path $u,n+1,v$.
It is easy to check that $L_1,...,L_{t-1},L_t', \{u,n\}, L_{t+1},...,L_r$ is a generalized ear decomposition $\E'$ of $H$ with $\varphi(\E') = \varphi(\E)-1$ (see Figure 6).

\begin{figure}[ht!]
\begin{tikzpicture}[scale=0.5] 

\draw [thick] (0,0) coordinate (a) -- (2,2) coordinate (b) ;
\draw [thick, dashed] (0,0) coordinate (a) -- (2,-2) coordinate (c) ;
\draw [thick] (4,0) coordinate (d) -- (2,2) coordinate (b) ;
\draw [thick] (4,0) coordinate (d) -- (2,-2) coordinate (c) ;
\draw [thick] (4,0) coordinate (d) -- (10,0) coordinate (e) ; 
\draw [thick] (-3,0) coordinate (f) -- (2,2) coordinate (b) ;
\draw [thick, dashed] (-3,0) coordinate (f) -- (2,-2) coordinate (c) ; 

\fill (a) circle (4pt);
\fill (b) circle (4pt);
\fill (c) circle (4pt);
\fill (d) circle (4pt);
\fill (f) circle (4pt);

\draw (2,2.7) node{$n+1$};
\draw (2,-2.7) node{$n$};
\draw (-3.2,0.7) node{$v$};
\draw (1,0) node{$w$};
\draw (4.2,0.7) node{$u$};
\draw (7.2,0.7) node{$L_t$};
\draw (3.5,-1.4) node{$L_t$};
 
\end{tikzpicture}
\caption{}
\end{figure}

The edges of $K \setminus H$ are of the form $\{w,n+1\}$, $w \in N(n)$, $w \neq u,v$. 
Adding these edges to the end of $\E'$ we obtain a generalized ear decomposition $\F$ of $K$ with 
$$\varphi(\F) = \varphi(\E') = \varphi(\E)-1.$$
 
Let $\b \in \NN^{n+1}$ such that $b_i -1$ is the number of even ears of $\E'$ which is different from the first ear and which starts with the vertex $i$, $i = 1,...,n+1$. Then $b_i = a_i$ for $i = 1,...,n-1$, $b_n = a_n-1$, $b_{n+1} = 1$.
Hence $|\b| = |\a|$.
By induction on $|\a|-n$, we may assume that $K^\b$ is a matching-covered graph with 
$|\b| = \varphi(\F) + n.$
From this it follows that 
$$|\a| = |\b| = \varphi(\F) + n = \varphi(\E) + n-1.$$
\end{proof}

\begin{Remark}
Proposition \ref{construction} provides a way to solve the diophantine system $(*)$. 
Let $U$ be an arbitrary dominating set $U$ of $G$ such that $G_U$ is a connected bipartite graph. Let $\a \in \NN^n$ such that $a_i = 0$ for $i \not\in U$ and $a_i -1$ is the number of even ears which starts with the vertex $i$ for $i \in U$.
By Proposition \ref{construction}, $G^\a$ is a matching-covered graph. Hence, $\a$ is a solution of $(*)$ by Theorem \ref{matching-covered}. We don't know whether all solutions $\a \neq 0$ of $(*)$ arises in this way.
\end{Remark}

The above propositions yield the following explicit formula for the minimal order of matching-covered parallelizations whose base graph is $G$.

\begin{Theorem} \label{min}
Let $G$ be a connected bipartite graph. Then 
$$\min\{|\a||\ G^\a \text{ is matching-covered with } \supp(\a) = V\} = \varphi(G)+n-1.$$
\end{Theorem}

\begin{proof}
By Proposition  \ref{order}, we have
$$\min\{|\a||\ G^\a \text{ is matching-covered with } \supp(\a) = V\} \ge  \varphi(G)+n-1.$$
To prove the converse inequality we choose a generalized ear decomposition $\E$ of $G$ with $\varphi(\E) = \varphi(G)$. Using Proposition \ref{construction} we can construct a matching-covered parallelization $G^\a$ with 
$\supp(\a) = V$ and $|\a| = \varphi(G)+n-1$. Therefore,
$$\min\{|\a||\ G^\a \text{ is matching-covered with } \supp(\a) = V\} \le \varphi(G)+n-1.$$
\end{proof}

\begin{Remark} 
Since $|\a|$ is the number of vertices of $G^\a$, $|\a|$ is an even number if $G^\a$ is matching-covered.
Therefore, $\varphi(G)+n-1$ is also an even number by Theorem \ref{min}.
\end{Remark}

%%%%%%%%%%%%%%%%%%%%%%%

\section{The main result}

Theorem \ref{min} leads us to the following invariant. For any connected graph $G$ on $n$ vertices we set 
$$\mu(G) := (\varphi(G) + n - 1)/2.$$
It follows that $\mu(G) \ge n/2$.

% For brevity, we call an induced subgraph $G_U$ {\it dominating} if $U$ is a dominating set of $G$.
Let $G$ be a connected bipartite graph.
If $G$ is not a complete bipartite graph, we define
$$s(G) := \min\{\mu(G_U)|\ G_U \text{ is a dominating connected induced subgraph of $G$}\}.$$ 
If $G$ is a complete bipartite graph, we set $s(G) = 0$.

\begin{Theorem} \label{main}
Let $I$ be the edge ideal of a connected bipartite graph $G$. Then  
$$\dstab(I) = s(G)+1.$$
\end{Theorem}

\begin{proof} 
By Theorem \ref{system}, we have
$$\dstab(I) = \min\{|\a|/2|\  \a \in \NN^n \text{ is a solution of } (*)\}+1.$$
By Lemma \ref{a=0}, $\a = 0$ is a solution of the system $(*)$ in Section 1 if and only if $G$ is a complete bipartite graph.
Therefore, $\dstab(I) = 1 = s(G)+1$ in this case.
 
If $G$ is not a complete bipartite graph, then 
\begin{multline*}
\{\a \in \NN^n|\ \a \text{ is a solution of } (*)\} =\\
\{\a \in \NN^n|\ \text{$\supp(\a)$ is a dominating set of $G$ and $G^\a$ is matching-covered}\}.
\end{multline*}
by Theorem \ref{matching-covered}.
Let $U = \supp(\a)$ for such an $\a$. 
Then $G_U$ is a dominating connected bipartite subgraph of $G$ by Corollary \ref{base}.
Applying Theorem \ref{min} to $G_U$, we have
$$\min\{|\a||\ G^\a \text{ is matching-covered with } \supp(\a) = U\} = \varphi(G_U) + |U| - 1.$$ 
Since $\mu(G_U) = (\varphi(G) + |U| - 1)/2,$ this implies
$$\min\{|\a|/2|\  \a \in \NN^n \text{ is a solution of } (*)\} = s(G).$$
Therefore, $\dstab(I) = s(G)+1$.
\end{proof}

% One may think that $\mu(G) \ge\mu(G_U)$ for some dominating connected bipartite proper subgraph $G_U$. However, there are cases such that $\mu(G) < \mu(G_U)$ for all dominating connected induced subgraph $G_U$.

% \begin{Example} 
% Let $G$ be the hexagon. Then $G_U$ is a dominating connected bipartite proper subgraph if and only if $U$ is a path of length 4 in $G$. All generalized ear decompositions of $G_U$ consist of five 2-cycles. Hence, $\mu(G_U) = 4$. On the other hand, $G$ has an ear decomposition consisting of only one ear, which is $G$. Other generalized ear decompositions of $G$ consists of 5 2-cycles and an edge path. Hence, $\mu(G) = 3$. Therefore, $s(G)=3$.
%\end{Example}

By Theorem \ref{main}, to compute $\dstab(I)$ we only need to compute $s(G)$.
As a first application we give a complete description of connected bipartite graphs with $\dstab(I) = 1,2,3,4$. 
Recall that a graph is spanned by a subgraph if they have the same vertex set.

\begin{Corollary} \label{small}
Let $G$ be a connected bipartite graph $G$. Then \par
{\rm (1)}  $\dstab(I) = 1$ if and only if $G$ is a complete bipartite graph, \par
{\rm (2)} $\dstab(I) = 2$ if and only if $G$ is not a complete bipartite graph and $G$ has a dominating set that is an edge, \par
{\rm (3)} $\dstab(I) = 3$ if and only if $G$ is not a complete bipartite graph, $G$ has no dominating set that is an edge, and $G$ has a dominating set whose induced subgraph is a 2-path or a 4-cycle. \par
{\rm (4)} $\dstab(I) = 4$ if and only if $G$ is not a complete bipartite graph, $G$ has no dominating set as in {\rm (2)} and {\rm (3)} and $G$ has a dominating set which is a 3-path or spanned by the union of a 4-cycle and a leaf edge or spanned by a 6-cycle.
\end{Corollary}

\begin{proof} 
We need to characterize the graphs $G$ with $s(G) = 0,1,2,3$.
By definition, $s(G) = 0$ means $G$ is a complete bipartite graph.
Assume that $G$ is not a complete bipartite graph. 

We have $s(G) = 1$ if and only if there is a dominating connected bipartite induced subgraph $G_U$ with $\mu(G_U) = 1$. It is clear that $\mu(G_U) = 1$ if and only if $|U| = 2$, which means that $G_U$ is an edge.

We have $s(G) = 2$ if and only if $G$ has no dominating set that is an edge and there is a dominating connected bipartite induced subgraphs $G_U$ with $\mu(G_U) = 2$. From this it follows that $3 \le |U| \le 4$.
If $|U| = 3$, then $G_U$ must be a 2-path, hence $\mu(G_U) = 2$. 
If $|U| = 4$, then $\mu(G_U) = 2$ if and only if $\varphi(G_U) = 1$, which means that $G$ is a 4-cycle.

We have $s(G) = 3$ if and only if $G$ has no dominating set as in (2) and (3) and there is a dominating connected bipartite induced subgraphs $G_U$ with $\mu(G_U) = 3$. From this it follows that $4 \le |U| \le 6$.
If $|U| = 4$, then $G_U$ must be a 3-path, hence $\mu(G_U) = 3$. 
If $|U| = 5$, then $\mu(G_U) = 3$ if and only if $\varphi(G_U) = 2$, which means that $G$ is spanned by the union of a 4-cycle and a leaf edge.
If $|U| = 6$, then $\mu(G_U) = 3$  if and only if $\varphi(G_U) = 1$, which means that $G$ is spanned by a 6-cycle.
We leave the reader to check the details. 
\end{proof}

\begin{Remark}
If $G$ is a complete bipartite graph $K_{r,s}$, $r, s \ge2$, then $G$ always has a minimal dominating set that is an edge. This explains why we need the condition $G$ is not a complete bipartite graph in (2).
\end{Remark}

Another immediate consequence of Theorem \ref{main} is the following formula of \cite{TNT}. Let $v(G)$ and $\varepsilon_0(G)$ denote the number of the vertices and leaf edges of $G$, respectively.

\begin{Corollary} \label{tree} \cite[Lemma 5.4]{TNT}
Let $G$ be a tree. Then 
$$\dstab(I) = v(G) - \varepsilon_0(G).$$
\end{Corollary}

\begin{proof} 
If $v(G) = 2$, then $\dstab(I) = 1 = v(G) - \varepsilon_0(G).$
If $v(G) > 2$, we consider an arbitrary dominating connected subgraph $G_U$. 
Since $G_U$ has no cycles, every generalized ear decomposition of $G_U$ consists of only 2-cycles.
Hence, $\varphi(G_U) = |U|-1$, which implies $\mu(G_U) = |U|-1$. 
Since $v(G) > 2$, 
the set of the vertices of degree $\ge 2$ is minimal among such $U$. Its size is $v(G)-\varepsilon_0(G)$. 
Therefore, 
$$s(G) = \min \mu(G_U) = v(G)-\varepsilon_0(G)-1.$$
\end{proof}

Now we are going to deduce some upper bounds for $\dstab(I)$ in terms of simple invariants of $G$.
For that we need the following observations. 

\begin{Lemma} \label{spanning} 
Let $G$ be a connected bipartite graph. Let $H$ be a connected spanning subgraph of $G$.
Then $\mu(G) \le \mu(H)$ and $s(G) \le s(H).$
\end{Lemma}

\begin{proof} 
Let $\F$ be a generalized ear decomposition of $H$ with $\varphi(\F) = \varphi(H)$.
Since the edges of $G - H$ have both ends in $H$, adding them as paths to $\F$ we obtain a generalized ear decomposition $\E$ of $G$ with $\varphi(\E) = \varphi(\F)$. Therefore, 
$$\mu(G) \le (\varphi(\E)+v(G)-1)/2 = (\varphi(\F)+v(G)-1)/2 = \mu(H).$$

Let $H_U$ be a dominating connected bipartite subgraph of $H$ such that $s(H) = \mu(H_U)$.
Since $G_U$ is also a dominating connected bipartite subgraph of $G$, $s(G) \le \mu(G_U)$.
Since $H_U$ is a connected spanning subgraph of $G_U$, $\mu(G_U) \le \mu(H_U).$ Therefore, $s(G) \le s(H)$.
\end{proof}

\begin{Lemma} \label{mu} 
Let $G$ be a connected graph. Then $\mu(G) \le v(G)-1.$
Equality holds if and only if $G$ is a tree.
\end{Lemma}

\begin{proof}
If $G$ is a tree, $\varphi(G) = v(G)-1$ by the proof of Corollary \ref{tree}.
If $G$ is not a tree, $G$ has a cycle $C$ of length $\ge 3$.
By Remark \ref{existence}, $C$ is the first ear of some generalized ear decomposition $\E$ whose subsequent ears are edges. The 2-cycles among these edges equals the number of vertices of $G-C$. Since $C$ may be an even cycle, $\varphi(G) \le v(G)-|C|+1 \le v(G)-2$. 
\end{proof}

\begin{Theorem} \label{bound}
Let $G$ be a connected bipartite graph. 
Let $C_1,...,C_r$ be disjoint cycles of $G$, $r \ge 1$.
Then 
$$\dstab(I) \le v(G)-\varepsilon_0(G)-(v(C_1)+\cdots+v(C_r))/2 + r.$$
Equality holds if these cycles are the only cycles of the graph and have lengths $\ge 6$.
\end{Theorem}

\begin{proof}
Let $G'$ be the graph obtained from $G$ by shrinking every cycle $C_i$ to a new vertex $c_i$.
Then there is a graph epimorphism $p: G \to G'$ that maps every vertex of $G \setminus \cup_{i=1}^rC_i$ to itself and the vertices of $C_i$ to $c_i$. Let $U'$ be the set consisting $c_1,...,c_r$ and all other vertices of $G'$ of degree $\ge 2$. Then $G'_{U'}$ is a dominating connected subgraph of $G'$. 

Let $T$ be a spanning tree of $G'_{U'}$. Then 
$$v(T) = v(G) -  \varepsilon_0(G) - (v(C_1)+\cdots+v(C_r)) + r.$$
For every edge of $T$ we choose a preimage in $G$. 
These preimages together with $C_1,...,C_r$ form a connected spanning subgraph $H$ of $G$ such that $C_1,...,C_r$ are the only cycles of $H$. We have
$$v(H) = v(G) -  \varepsilon_0(G).$$

Let $\F$ be a generalized ear decomposition of $T$. 
Then $\F$ consists of only 2-cycles with $|\F| = v(T)-1$.
Let $\E$ be a generalized ear decomposition of $H$ consisting of the preimages of the edges of $\F$ in $H$ and the cycles $C_1,...,C_r$ in some suitable order. Then 
$\varphi(\E) = |\F| + r = v(T) + r-1.$  Therefore,
\begin{align*}
\mu(H) & \le (\varphi(\E) + v(H) -1)/2 \le (v(T) + v(H) + r-2)/2 \\
& = v(G) -  \varepsilon_0(G) - v(C_1)+\cdots+v(C_r))/2 + r-1.
\end{align*}
By Lemma \ref{spanning}, $s(G) \le s(H) \le \mu(H)$. Hence we obtain the bound of the theorem.

Assume that $C_1,...,C_r$ are the only cycles of $G$ and $v(C_i) \ge 6$, $i = 1,...,r$. We will prove that the bound of the theorem is attained. 

Let $G_U$ be a dominating connected induced subgraph of $G$ with $\mu(G_U) = s(G)$. Let $U_i$ be the set of vertices of $C_i$ in $U$.  
Since $G_{U_i}$ is a dominating connected induced subgraph of $C_i$, $G_{U_i}$ is a path and $|U_i| \ge v(C_i)-2.$  

Let $\E$ be a generalized ear decomposition of $G_U$ with $\varphi(\E) = \varphi(G_U).$
The edges of $G_U$ not contained in $C_1,...,C_r$ are 2-cycles of $\E$ because they are not contained in any cycle.
Therefore, the ears of $\E$ which involve only vertices of $C_i$ form a generalized ear decomposition of $H_{U_i}$.
Let $\E_i$ denotes this decomposition. 
If $U_i = C_i$, then $\varphi(\E_i) = 1$. Hence $\varphi(\E_i) + |U_i| = v(C_i) + 1.$
If $|U_i| = v(C_i)-1$, then $\varphi(\E_i) = v(C_i) -2$. Hence $\varphi(\E_i) + |U_i| = 2v(C_i) - 3 \ge v(C_i) +1.$
If $|U_i| = v(C_i)-2$, then $\varphi(\E_i) = v(C_i)-3$. Hence $\varphi(\E_i) + |U_i| = 2v(C_i) - 5 \ge v(C_i) +1$ because $v(C_i) \ge 6$.
So we always have $\varphi(\E_i) + |U_i| \ge v(C_i) + 1$, $i = 1,...,r$.

By the definition of the map $p$ we have
$$|U| = |p(U)| +  |U_1| + \cdots +  |U_r| - r.$$
Since the number of the edges not contained in $C_1,...,C_r$ is $|p(U)|-1$, 
$$\varphi(\E) = |p(U)|-1 + \varphi(\E_i) + \cdots + \varphi(\E_r).$$
From this it follows that 
\begin{align*}
s(G) & = (\varphi(\E) + |U|-1)/2\\
&= \big(2|p(U)|  + \varphi(\E_1) + \cdots + \varphi(\E_r) +  |U_1| + \cdots +  |U_r| -r -2\big)/2\\
& \ge |p(U)| + (v(C_1) + \cdots + v(C_r))/2 - 1.
\end{align*}
Since $C_1,...,C_r$ are the only cycles of $G$, $G'$ is a tree.
Since $G'_{p(U)}$ is a dominating connected subgraph of $G'$, $p(U)$ contains all vertices of degree $\ge 2$ of $G'$. Since $U$ contains at least a vertex of each cycle $C_i$,
$p(U)$ contains all vertices $c_i$. From this it follows that 
$|p(U)| \ge v(G) - \varepsilon_0(G)$. Therefore,
$$s(G) \ge v(G) - \varepsilon_0(G) + (v(C_1) + \cdots + v(C_r))/2 + r-1.$$
Hence the bound of the theorem is attained. 
\end{proof}

The case $r = 1$ of Theorem \ref{bound} immediately yields the following result of \cite{TNT}, whose original proof is very complicated.

\begin{Corollary} \label{unicyclic} \cite[Proposition 3.4 and Lemma 5.3]{TNT}
Let $G$ be a connected bipartite graph. Let $2k$ be the length of its longest cycle. Then 
$$\dstab(I) \le v(G) -\varepsilon_0(G)-k+1.$$
Equality holds if $G$ is a unicyclic graph and $k \ge 3$.
\end{Corollary}

Equality does not hold in Corollary \ref{unicyclic} if we drop the assumption $k \ge 3$.

\begin{Example}
If $G = C_4$, then $\dstab(I) = 1$ by Corollary \ref{small}(1). If $G$ is the union of $C_4$ with a leaf edge, then $\dstab(I) = 2$ by Corollary \ref{small}(2). In both cases, $v(G) -\varepsilon_0(G)-k+1 = 3$.
\end{Example}

A graph is called  {\it well-covered} (or unmixed) if every minimal vertex cover has the same size \cite{Pl1,Vi}. 
In algebraic terms, this condition means that the edge ideal is an unmixed ideal, i.e. the minimal primes have the same height. 
The order of a well-covered bipartite graph must be an even number. 
The following bound for $\dstab(I)$ is a main result of \cite{BM}. This bound can be easily proved by Theorem \ref{main}. 
Moreover, we can characterize those graphs for which this bound is attained.

\begin{Theorem} \label{well-covered}
Let $G$ be a connected well-covered bipartite graph on $n$ vertices. Then 
$$\dstab(I) \le n/2.$$
Equality holds if and only if $G$ is a tree with $n/2$ leaves.
\end{Theorem}

\begin{proof}
By Corollary \ref{small}(1), we may assume that $G$ is not a complete bipartite graph.
From this it follows that $n \ge 4$. By \cite[Theorem 3.1]{Pl2}, $G$ has a perfect matching $M$ such that for every edge $\{u,u'\}$ in $M$, the induced subgraph of the neighbors of $u$ and $u'$ forms a complete bipartite graph. We claim that for any vertex $u$ with $\deg_Gu \ge 2$, there is a dominating connected subgraph $G_U$ with $n/2$ vertices containing $u$ such that $\deg_Gv \ge 2$ for any $v \in U$ and each edge of $M$ contains only a vertex of $U$. As a consequence, $\mu(G_U) \le n/2-1$ by Lemma \ref{mu}. Hence 
$$\dstab(I) = s(G) + 1 \le \mu(G_U)+1 \le n/2.$$

If $n = 4$, $G$ must be a path of length 3 and $M$ is the set of the two outer edges of this path. In this case, $U$ is the set of the two inner vertices of this path.  

If $n > 4$, let $u$ be an arbitrary vertex with $\deg_Gu \ge 2$.
Let $u'$ be the adjacent vertex of $u$ in $M$ and $G' := G-u-u'$.
Let $G'_1,...,G_c'$ be the connected components of $G'$. By the above property of $G$, every $G_j'$ must contain an adjacent vertex of $u$.
Let $M_j$ be the restriction of $M$ in $G_j'$, $j = 1,...,c$.
Then $M_j$ is a perfect matching of $G_j'$ such that for every edge $\{v,v'\}$ in $M_j$, the induced subgraph of $G'$ on the union of the neighbourhoods of $v$ and $v'$ forms a complete bipartite graph. 
By \cite[Theorem 3.1]{Pl2}, this implies that $G_i'$ is also a connected well-covered bipartite graph. 
We will show that every $G_j$ has a dominating connected subgraph $G_{U_j}$ with $v(G_j)/2$ vertices which contain a vertex adjacent to $u$. 

Let $v$ be an adjacent vertex of $u$ in $G_j$. If $v(G_j) = 2$, we choose $U_j = \{v\}$.
If $v(G_j) \ge 4$ and $\deg_{G_j} v = 1$, let $v'$ be the adjacent vertex of $v$ in $M_j$. Then $v'$ must be adjacent to another vertex $w$ of $G_j$. By \cite[Theorem 3.1]{Pl2}, $w$ is adjacent to $u$.
Replacing $v$ by $w$, we may assume that $\deg_{G_j} v \ge 2$. 
Using induction we may assume that $G_j'$ has a dominating connected subgraph
$G_{U_j}$  with $v(V_j)/2$ vertices containing $v$ such that  $\deg_{G_j}w \ge 2$ for any $w \in U_j$ and each edge of $M_j$ contains only a vertex of $U_j$. 

Let $U = U_1 \cup \cdots \cup U_c + u$. Then $G_U$ is clearly a dominating connected subgraph of $G$ with $n/2$ vertices such that $\deg_Gv \ge 2$ for any $v \in U$ and each edge of $M$ contains only a vertex of $U$. This completes the proof of the claim.

Now we consider the case $\dstab(I) = n/2$. Let $G_U$ be a dominating dominating connected subgraph as in the above proof.
Then $\mu(G_U) = n/2-1$. By Lemma \ref{mu}, 
$G_U$ is a tree. By the above claim, the vertices of $U$ have degree $\ge 2$ and each edge of $M$ contains only a vertex of $U$. 
Therefore, $G$ is a tree with $n/2$ leaves if we can show that every edge of $M$ has a leaf. We may assume that $n > 4$.

Let $u$ be an arbitrary vertex of degree $\ge 2$. It suffices to show that its adjacent vertex $u'$ in $M$ is a leaf. Let $G' := G-u-u'$. 
If $u'$ is not a leaf, there exists an adjacent vertex $v$ of $u'$ in $G'$. Since $v$ is adjacent to any vertex adjacent to $u$, $G'$ is connected. As in the proof of the claim, there is a dominating connected subgraph $G_{U'}$ of $G'$ with $n/2-1$ vertices which contains a vertex adjacent to $u$.  If $v \not\in U'$, then $v' \in U'$. Let $w$ be an adjacent vertex of $v'$ in $U'$.  By \cite[Theorem 3.1]{Pl2}, $w$ is an adjacent vertex of $u'$ in $G'$. Replacing $v$ by $w$, we may assume that $v \in U'$. Then $u,u'$ are adjacent to vertices in $U'$. Hence $G_{U'}$ is a dominating connected subgraph of $G$. As a subgraph of $G_U$, $G_{U'}$ is a tree. This implies $\mu(G_{U'}) = n/2-2$ by Lemma \ref{mu}. Therefore, $s(G) \le \mu(G_{U'}) = n/2-2$, a contradiction to the assumption that $\dstab(I) = n/2$.

Conversely, if $G$ is a tree with $n/2$ leaves, where $n$ is even, then $\dstab(I) = n-n/2 = n/2$ by Corollary \ref{tree}.
This completes the proof of Theorem \ref{well-covered}.
\end{proof} 

\begin{Remark} The notion of $\mu(G)$ was introduced in coding theory \cite{SZ}. 
It was defined as the maximum cardinality of a join of $G$.
A subset $S$ of edges of $G$ is called a {\it join} if $|C \cap S| \le |C|/2$ holds for every circle $C$ of $G$.
It was proved in \cite[Main Theorem]{Fr} that if $G$ is a 2-edge connected graph, then 
$\mu(G) = (\varphi(G)+v(G)-1)/2$, where $\varphi(G)$ is the minimum number of even ear in ear decompositions of $G$. By Lemma \ref{phi}, this definition of $\varphi(G)$ is the same as by means of generalized ear decompositions.
\end{Remark}

%%%%%%%%%%%%%%%%%%%%%

\section{A general formula for $\dstab(I)$}

In this section, we deduce from Theorem \ref{main} a general formula for $\dstab(I)$ when $G$ is an arbitrary graph. For that we need an explicit formula for $\dstab(I)$ when $G$ is a connected nonbipartite graph. 
Such a formula can be derived from the study on the associated primes of powers of edge ideals in \cite{LT}.

Let $G_1,...,G_c$ be the connected components of $G$.
A generalized ear decomposition of $G$ is a sequence of generalized ear decompositions of $G_1,...,G_c$.
If $G$ is a strongly nonbipartite graph, every $G_i$ has an odd cycle. 
By Remark \ref{existence}, $G$ has generalized ear decompositions which starts with an odd cycle in each $G_i$.
Let $\varphi^*(G)$ denote the minimum number of even ears in such generalized ear decompositions of $G$. 
Set
$$\mu^*(G) := (\varphi^*(G)+n-c)/2.$$

If $G$ is a connected nonbipartite graph, we denote by $s(G)$ the minimum number of $\mu^*(G_U)$ among dominating strongly nonbipartite subgraphs $G_U$ of $G$. Note that $G_U$ needs not be a connected graph. 

\begin{Theorem} \label{nonbipartite} 
Let $G$ be a connected nonbipartite graph. Then
$$\dstab(I) = s(G) + 1.$$
\end{Theorem}

\begin{proof} 
By \cite[Lemma 3.4]{CMS}, $\displaystyle \lim_{t \to \infty}\depth R/I^t = 0$.
We have $\depth R/I^t = 0$ if and only if $\mm$ is an associated prime of $I^t$.
By \cite[Theorem 3.6]{LT}, $\mm$ is an associated prime of $I^t$ if and only if $t \ge s(G)+1$. 
Therefore, $\dstab(I) = s(G)+1$.
\end{proof}

Since $I$ is unmixed, we always have $\depth R/I > 0$. Hence $\dstab(I) > 1$ if $G$ is a connected nonbipartite graph.
Using results of \cite[Section 5]{LT} we can describe those graphs with $\dstab(I) = 2,3$. We refer to \cite{LT} for more details and other applications of Theorem \ref{nonbipartite}.

\begin{Remark}
The definition of $\varphi^*(G)$ in \cite{LT} is based on a slightly different definition of generalized ear decomposition.
There, a generalized ear decomposition of $G$ is defined as a sequence of ears such that the first ear is a cycle, the endpoints of each subsequent ears are the only vertices that belong to earlier ears, and the ears pass through all vertices of the graph. This decomposition needs not involve all edges of $G$. Adding the remaining edges of $G$, 
we obtain a generalized ear decomposition of $G$ in the sense of this paper. Since these edges have both endpoints in the ears of the former decomposition, they are paths of length 1. Therefore, the numbers of even ears in both generalized ear decompositions are the same. If we define $\varphi^*(G)$ as the minimum number of even ears in generalized ear decompositions in the sense of \cite{LT} which start with an odd cycle in each connected component of $G$, then $\varphi^*(G)$ is the same as in the above definition.
\end{Remark}

Combining Theorem \ref{main} and Theorem \ref{nonbipartite} we obtain the following general formula for the index of depth stability of edge ideals of any graph.

\begin{Theorem} \label{arbitrary}
Let $G$ be an arbitrary graph without isolated vertices. Let $G_1,...,G_c$ be the connected components of $G$. Then
$$\dstab(I) = s(G_1) + \cdots + s(G_c)+1.$$
\end{Theorem}

\begin{proof} 
Let $I_j$ be the edge ideal of $G_i$, $i = 1,...,c$.
By Theorem \ref{general}, 
$$\dstab(I) = \dstab(I_1) + \cdots + \dstab(I_c) - c +1.$$
By Theorem \ref{main} and Theorem \ref{nonbipartite}, we have $\dstab(I_i) = s(G_i)+1$, 
which is independent of whether $G_j$ is bipartite or non-bipartite. From this it follows that
$$\dstab(I) = s(G_1) + \cdots + s(G_c)+1.$$
\end{proof}

\begin{Remark} 
The approach in the nonbipartite case is completely different than that in the bipartite case.
In the bipartite case, we have to study the vanishing of the first local cohomology module and come over to the existence of matching-covered parallelizations. In the nonbipartite case, the socle module is studied leading to the existence of matching-critical parallelizations \cite{LT}.  Recall that a connected graph is {\it matching-critical} if the subgraph obtained by the removal of any vertex has a perfect matching \cite{Lo1,LP}. It happens that both matching-covered and matching-critical graphs can be characterized in terms of ear decompositions. We are unable to find a unified approach to the bipartite and nonbipartite cases. The first step is to find a common combinatorial meaning for $\mu(G)$ and $\mu^*(G)$.
\end{Remark}

\medskip

\noindent {\bf Acknowledgement}.  
This paper was completed during a research stay of the second author at Vietnam Institute of Advanced Study in Mathematics. The authors are supported by Project NCXS02.01/22-23 of Vietnam Academy of Science and Technology. They are thankful to the referees for many suggestions which correct some proofs and improve the presentation of this paper.

%%%%%%%%%%%%%%%%%%

\end{document}